\documentclass[11pt]{article}
\usepackage{graphics}
\usepackage{graphicx}
\usepackage{epsfig}
\usepackage{amsmath}
\DeclareMathOperator{\sign}{sign}
\usepackage{amsfonts}
\usepackage{amssymb}
\usepackage{xcolor}
\usepackage{enumerate}
\usepackage{subfigure}
\usepackage[export]{adjustbox}
\makeatletter
\@addtoreset{equation}{section}
\makeatother
\usepackage{pst-plot}
\psset{xunit=15mm}

\numberwithin{equation}{section}
\usepackage{tikz}
\makeatletter
\@namedef{subjclassname@2020}{\textup{2020} Mathematics Subject Classification}
\makeatother

\usepackage{a4wide,amsmath,amssymb,latexsym,amsthm}

\usepackage{enumitem}

\usepackage[colorlinks=true,urlcolor=blue,
citecolor=red,linkcolor=blue,linktocpage,pdfpagelabels,
bookmarksnumbered,bookmarksopen]{hyperref}
\usepackage[english]{babel}

\usepackage{setspace}

\newcommand{\1}{\mathbbm{1}}

\usepackage{authblk}    

\numberwithin{equation}{section}
\usepackage{mathrsfs,mathtools,epic,bm}
\usepackage{hyperref}
\hypersetup{colorlinks=true,linkcolor=blue,filecolor=mangeta,urlcolor= cyan}

\newtheorem{theorem}{Theorem}[section]
\newtheorem{proposition}[theorem]{Proposition}
\newtheorem{lemma}[theorem]{Lemma}
\newtheorem{corollary}[theorem]{Corollary}
\theoremstyle{definition}

\newtheorem{remark}[theorem]{Remark}

\def \dis {\displaystyle}

\newcommand{\sgn}{\operatorname{sign}}

\def \R {\mathbb{R}}

\def \B {\mathbf{B}}
\def \M {\mathcal{M}}

\def \C {\mathbf{C}}

\def \F {\mathbf{F}}
\def \D {\mathbf{D}}
\def \M {\mathcal{M}}
\def \L {\mathcal{L}}
\def \BUC{\mathrm{BUC}}
\def \V {\mathcal{V}}

\def \K {\mathcal{K}}
\def \c{c_{k_1,k_2,r,d}}

\numberwithin{equation}{section}

\makeatletter
\renewcommand{\@seccntformat}[1]{\csname the#1\endcsname.\quad}
\makeatother

\def \dis {\displaystyle}
\title{Monotonicity of the propagation speed with respect to the diffusion in a Lotka-Volterra competition-diffusion system }
\author[]{Cyrille Kenne\thanks{Corresponding author}}

\affil[]{\footnotesize Department of Mathematics, University of British Columbia, Vancouver, BC V6T 1Z2, Canada}

\date{\today}
	
\begin{document}
\maketitle
\footnotetext{Email address: \href{mailto:kenne@math.ubc.ca}{kenne@math.ubc.ca}}

\begin{abstract}
In this paper, we study the dependence of the propagation speed on the diffusion ratio in a Lotka-volterra competition-diffusion system under strong competition.  The model is known to admits a unique monotone travelling front connecting two exclusion equilibria and the sign of its speed determines which species invades the territory occupied by the other.  We establish smoothness results for the wave speed and the wave profile. Our main result shows that the wave speed is a strictly decreasing function of the diffusion ratio on explicit unbounded  regions of parameter space. This monotonicity had been observed numerically but, had not previously been proved. The proof combines the smooth dependence of both the travelling front and its speed on the diffusion ratio, an adjoint identity for the linearized operator and a maximum principle argument. At zero speed, we also prove that there exists a smooth threshold, we obtain an exact formula for its derivative and deduce its monotonicity. These threshold results also provide a complete characterization of the sign of the wave speed.
\end{abstract}

\textbf Mathematics Subject Classification. {35C07, 35K55,	35K57, 92D25, 92D40 }\par
\noindent
{\textbf {Key-words}}~:~  Lotka-Volterra; travelling wave; monotonicity; competition-diffusion model.

\section{Introduction}
In this paper, we  consider the following Lotka-Volterra system  
\begin{equation}\label{model}
\left\{ 
\begin{array}{llllll}
	u_t&=&u_{xx}+u(1-u-k_1v), & \text{ in } \; (0,\infty)\times \R,\\
	v_t&=&dv_{xx}+rv(1-v-k_2u),& \text{ in } \; (0,\infty)\times \R,
\end{array} 
\right.
\end{equation}
 where $k_1,k_2>1$, $d>0$ and $r>0$. Here, the parameter $d$ is the diffusion coefficient of the species $v$ relative to the species $u$ and $r$ is its intrinsic growth rate relative to that of $u$. The parameters $k_1$ and $k_2$ represent the interspecific competition coefficients.
 
 Spatial dispersal can affect not only the rate at which a population spread but also the outcome of the competition between two species. In the strong competition case  ($k_1,k_2>1$), it is well known that the system \eqref{model} has two locally stable exclusion equilibria  $(1,0)$ and $(0,1)$. The interface separating the territories of species $u$ and $v$ can be then described by a bistable travelling wave.  The sign of the wave speed determines which species invades the other, while its absolute value give the rate at which this invasion occurs.
 
Throughout the paper, we adopt the following notion of monotone travelling wave. A monotone travelling front solution to \eqref{model}  with speed $c$ is a solution of the form,
 \begin{equation*}
 	(u,v)(t,x)=(U,V)(\xi), \qquad \xi=x-ct,
 \end{equation*}
such that
 \begin{equation}\label{eq1}
 	\begin{array}{llllll}
 	U''+cU'+U(1-U-k_1V)&=&0,\\
 	dV''+cV'+rV(1-V-k_2U)&=&0, 
 	\end{array} 
 \end{equation}
 with 
 \begin{equation}\label{eq2}
 	(U,V)(-\infty)=(1,0), \qquad (U,V)(+\infty)=(0,1) 
 \end{equation}
 and 
  \begin{equation}\label{eq2a}
 U'<0<V' \; \text{ in } \R.
 \end{equation}

Here, $c<0$ means that the species $v$ invades species $u$ and $c>0$ means that the species $u$ invades the species $v$. 
 
The classical theory of travelling waves for bistable competition-diffusion systems has been developed in \cite{conley1984, gardner1982, kan1995, kan1996}.  In particular, it well known that for every $k_1,k_2>1$ and $r,d>0$, there exists a unique real number $c:=c_{k_1,k_2,r,d}$ for which \eqref{eq1}-\eqref{eq2a} admits a solution $(U,V):=(U_{k_1,k_2,r,d}, V_{k_1,k_2,r,d})$. This solution belongs to $C^\infty(\R)^2$, is unique up to a translation and satisfies $0<U(\xi),V(\xi)<1$ for every $\xi\in\R$.  It is also known that if we normalize the front by assuming that  $U(0)=\frac{1}{2}$, then the front is unique \cite{girardin2015}. In this paper, we will consider such normalized front.

Determining the sign of the wave speed has been extensively studied in  \cite{alfaro2023, chang2023, chen2026, girardin2015, guo2013, ma2019, morita2023,  rodrigo2000,rodrigo2001, risler2017, xiao2025} and the references therein. A recent work \cite{nakamura2026} substantially enlarged the regions of parameters in which  previously results were known. We also refer to \cite{girardin2019} for an excellent survey on the topic. For related results on spreading speeds in Lotka-Volterra systems, see \cite{carrere2018, peng2021}. 

We mention that when completing this manuscript, a preprint \cite{ma2026} appeared. Its authors obtained the continuity of the zero speed threshold, its strict monotonicity, endpoint values and the complete propagation direction for the system \eqref{model}.  Thus, these results overlap with the characterization of the zero speed threshold presented here. In addition, we establish the smoothness of the threshold and derive an explicit formula for its derivative (see \eqref{derivK}). Thereby providing additional information on the threshold mentioned in \cite[Remark 5.6]{ma2026}. 

In our previous work \cite{kenne2026}, we proved that when $r=1$ and $k_1=k_2:=k>1$,
\begin{equation*}
	\sign c_{k,k,1,d}=\sign(1-d),
\end{equation*}
thereby establishing the ``Unity is not strength'' theorem (see e.g., \cite{girardin2019, girardin2018}): under equal intrinsic growth rates and strong competition, the faster diffuser prevails. As a first observation, we can adapt the scaling argument of \cite{guo2013}, to obtain the following result for every $r>0$ and for $k_1=k_2:=k>1$,
\begin{equation*}
	\sign c_{k,k, r,d}=\sign(r-d).
\end{equation*}
Consequently, species $v$ invades species $u$ when $d>r$, species $u$ invades species $v$ when $d<r$ and the front is standing when $d=r$. This means that, the direction of propagation is determined by the  ratio $\frac{d}{r}$. This determines the direction of propagation, but not how the signed speed varies with the diffusion ratio $d$. 

Kan-On \cite{kan1995} established monotone dependence of the propagation speed with respect to the reactions parameters. They obtained that the speed is strictly decreasing in $k_1$ and strictly increasing in $k_2$. The main idea was to differentiate the travelling-wave system with respect to the parameters under consideration and to test the resulting linearized equation against a bounded solution of the adjoint problem. They thus obtained an integral representation for the derivative of the wave speed and used the monotonicity properties of the wave profile as well as sign properties of the adjoint egeinfunction to determine the sign of this derivative.  However, their results do not cover the monotonicity with respect to the diffusion ratio $d$. 

In \cite{alzahrani2012}, Alzahrani et al. studied how relative motility and competition strength interact to determine the direction of bistable travelling wave in Lotka-Volterra  competition systems. They identified regions in which the motility rate can slow, halt or reverse  the invasion.  Their numerical simulations suggested that the wave speed is increasing with respect to the relative motility. Based on that, they conjectured that the wave speed is an increasing function of the relative motility.  We mention that their convention is opposite to ours and so, in our setting, their conjecture means the wave speed is decreasing with respect to $d$. Girardin \cite{girardin2019}  also observed a form of monotonicity of the speed  with respect to $d$ in numerical computations. We also mention \cite{hosono1998}. Risler \cite{risler2017}, developed a first-order variation for the speed under small perturbations of the diffusion coefficient.  Despite the numerical observations, no previous result has established the strict monotonicity of the bistable wave speed on the diffusion ratio in the Lotka-Volterra competition system \cite{girardin2019, girardin2015, nakamura2026}. 

One of the main purposes of this paper is to study the monotonicity of the wave speed. We prove that the wave speed is strictly decreasing with respect to the diffusion ratio $d$ on explicit unbounded parameters ranges. Our main result is the following.
\begin{theorem}\label{maintheorem}
	Let $k_1,k_2>1$ and $r,d>0$. Suppose that 
\begin{equation}\label{condition0}
	\c\leq 0	 \; \text{ and }\; \c\left(1-\frac{1}{d}\right)\leq 0.
\end{equation}
	Then 
	\begin{equation}\label{eq5}
		\partial_d \c <0.
	\end{equation}             
In particular, \eqref{eq5} holds true at every standing front.
\end{theorem}

In the symmetric case, we can obtain the exact parameters regions in which the speed is monotone with respect to the diffusion ratio $d$. 
\begin{corollary}\label{corollary1}
Let $k_1=k_2:=k>1$ and $r,d>0$. Suppose that either  $(i)$ $d>r$ and $d\geq 1$ or $(ii)$ $d=r$ holds. Then 
	\begin{equation}\label{eq5a}
	\partial_d c_{k,k,r,d}<0.
	\end{equation}             
This means that when $r\geq1$, the function $d\mapsto c_{k,k,r,d}$ is strictly decreasing on $[r, \infty)$ and when $0<r<1$,  the function $d\mapsto c_{k,k,r,d}$ is strictly decreasing on  $[1,\infty)$ and in addition,
\begin{equation*}
\left. \partial_d c_{k,k,r,d} \right|_{d=r}<0.
\end{equation*}                                                                                          
\end{corollary}
The proof of Theorem \ref{maintheorem} combines the differentiability of the travelling wave with respect to $d$ and an adjoint identity for the linearized operator as in \cite{kan1995}. The main difficulty is that the differentiation of the second equation in \eqref{eq1} with respect to $d$ gives a term involving $V''$ whose sign is not fixed. We exploit the properties of the system as well as those of the adjoint system and a maximum principle argument. To our knowledge, this is the first analytic proof of strict monotonicity with respect to $d$ on an unbounded interval. The global monotonicity remains open. 

We also provide a characterization of the zero speed set. The existence, continuity, strict monotonicity and endpoint values stated below are also in \cite[Theorem 1.3]{ma2026}. Our proofs are independent of the proofs of \cite{ma2026}. In addition, we obtain  the $C^\infty$ regularity and the exact formula \eqref{derivK}. 
\begin{theorem}\label{maintheorem1}
	Let $k_2>1$. There exists a unique function $\K(\cdot;k_2)$ such that
	\begin{equation}\label{zeroarie}
	\c=0\Longleftrightarrow k_1=\K\left(\frac{d}{r};k_2\right).
	\end{equation}
It is of class $C^\infty$, strictly decreasing on $(0,\infty)$ and satisfies:
	\begin{align}
		&\sqrt{k_2}<\K(\delta;k_2)<k_2^{2}\quad\text{for every }\delta>0,\qquad \K(1;k_2)=k_2,
		\label{bound1}\\
		&\lim_{\delta\downarrow0}\K(\delta;k_2)=k_2^{2},\qquad
		\lim_{\delta\to+\infty}\mathcal K(\delta;k_2)=\sqrt{k_2}.\label{limit1}
	\end{align}
	More precisely, let  $(U,V)$ be the standing monotone front at
	$(k_1,k_2,r,d)=(\K(\delta;k_2),k_2, 1,\delta)$. We set $p=-U'>0$, $q=V'>0$  and we consider
	$\Psi=(-a,b)^T$ with $a,b>0$ as given in Proposition \ref{propfredholm}. Then
	\begin{equation}\label{derivK}
	\partial_\delta\K(\delta;k_2) =-\frac{\dis\int_\R Wd\xi}{2\delta\dis\int_\R aUVd\xi}<0,
	\end{equation}
	where $W:=a'p-ap'=\delta(bq'-b'q)>0.$
The function $\K(\cdot; k_2)$ is therefore a decreasing bijection from $(0,\infty)$ onto $\bigl(\sqrt{k_2},k_2^{2}\bigr)$.
\end{theorem}
We mention that the bounds in \eqref{bound1} hold at every standing front. 
We also obtain the following complete sign characterization, which is also contained in \cite[Corollary 1.4]{ma2026}.
\begin{theorem}\label{maintheorem2}
	Let $k_1,k_2>1$ and $d,r>0$. We set $\delta=\frac{d}{r}$. 
	\begin{itemize}
		\item [$(i)$] If $k_1\le\sqrt{k_2}$ then $\c>0$, and if $k_1\ge k_2^{2}$ then $\c<0$, for every $d,r>0$.
		\item[$(ii)$] If $\sqrt{k_2}<k_1<k_2^{2}$, then there exists a unique $\delta^*=\delta^*(k_1,k_2)>0$ such
		that
		\begin{equation}\label{eq6}
			\c\begin{cases}>0,&0<\delta<\delta^*,\\ =0,&\delta=\delta^*,\\ <0,&\delta>\delta^*,\end{cases}
		\end{equation}
		and moreover
		\begin{equation}\label{eq7}
			\sgn(\delta^*-1)=\sgn(k_2-k_1).
		\end{equation}
	\end{itemize}
\end{theorem}

Combining Theorems \ref{maintheorem} and \ref{maintheorem2}, we obtain the following result.
\begin{corollary}\label{corollarymonotonegeneral}
	Let $k_1,k_2>1$ and $r>0$. 
		\begin{itemize}
		\item[$(i)$] If $k_1\geq k_2^2$, then the function $d\mapsto c_{k_1,k_2,r,d}$ is strictly decreasing on $[1,\infty)$.
		\item[$(ii)$] Assume that $\sqrt{k_2}<k_1<k_2^2$ and let $d^*:=r\delta^*(k_1,k_2)$.		If $d^* \geq 1$, then the function $d\mapsto c_{k_1,k_2,r,d}$ is strictly decreasing on $[d^*,\infty)$. If $d^*<1$, then it is strictly decreasing on $[1,\infty)$ and
		\begin{equation*}
			\left.\partial_dc_{k_1,k_2,r,d}\right|_{d=d^*}<0.
		\end{equation*}
	\end{itemize}
\end{corollary}

The rest of the paper is organized as follows. In Section \ref{speedsign}, we study the properties of the linearized operator associated to the system \eqref{eq1}-\eqref{eq2} and recall the standing front identities that we use in the sequel.  We prove our main results in Section \ref{monotonicity}. 

\subsubsection*{Notations}Throughout the paper, we use the standard Sobolev space notations. We write $L^2(\R)^2:=L^2(\R;\R^2)$ and $H^2(\R)^2:=H^2(\R;\R^2)$. We denote by $C_b(\R)^2:=C_b(\R;\R^2)$ the space of bounded continuous functions from $\R$ to $\R^2$, and by $\BUC:=\BUC(\R;\R^2)$ the space of bounded uniformly continuous functions from $\R$ to $\R^2$.  For $z\in\R^2$, $\|z\|$ denotes the Euclidean norm in $\R^2$. For $f,g\in L^2(\R)^2$, we denote their $L^2$-inner product by $\langle f,g\rangle_{L^2}$.  For a linear operator $T$, we write $\ker T$ and $\operatorname{Ran}T$ for its kernel and range, respectively and $\operatorname{span}\{h\}$ for the linear space generated by $h$.  $D\F(\Phi)$ denotes the Jacobian matrix of $\F$ evaluated at $\Phi$. The symbol $'$ will be used for the derivative with respect to the variable $\xi$ and $\partial_{a}$ for the partial derivative with respect to the parameter $a$.  Finally, we set $\pi:=(k_1,k_2,r,d)\in Q:=(1,\infty)^2\times(0,\infty)^2$.


\section{Properties of the front, linearization and the sign of the wave speed}\label{speedsign}
We recall in this section some known results (see e.g., \cite{kan1995,kan1996,risler2017}) and we study the properties of the linearized operator associated to the system \eqref{eq1}-\eqref{eq2}. 
We write  \eqref{eq1}-\eqref{eq2} on a compact form as follows. Let $\Phi:=(U,V)^T$, then \eqref{eq1}-\eqref{eq2} becomes
\begin{equation}\label{compactform}
	\D_d \Phi''+c \Phi'+\F_{k_1,k_2,r}(\Phi)=0, \qquad \Phi(-\infty) = E_- := \begin{pmatrix}1\\0\end{pmatrix}, \quad \Phi(+\infty) = E_+ := \begin{pmatrix}0\\1\end{pmatrix}, 
\end{equation}
where
\begin{equation}\label{compactform1}
	\D_d = \begin{pmatrix}1&0\\0&d\end{pmatrix}
	\qquad\text{and}\qquad
	\F_{k_1,k_2,r}(\Phi) = \begin{pmatrix} U(1-U-k_1V) \\ rV(1-V-k_2U)\end{pmatrix}.
\end{equation}
We also introduce the linearized operator of \eqref{compactform} at $\Phi$ defined by
\begin{equation}\label{linearizedoperator}
\L: H^2(\R)^2 \to L^2(\R)^2, \qquad \L(h):=\D_dh''+ch'+D\F_{k_1,k_2,r}(\Phi(\cdot))h. 
\end{equation}
Note that for $\Phi=(U,V)^T$ such that $0<U,V<1$, $\L$ is a well defined. The adjoint of the $L^2$-unbounded realization of $\L$ is defined as 
\begin{equation}\label{linearizedoperator}
	\L^*: H^2(\R)^2 \to L^2(\R)^2, \qquad \L^*(h):=\D_dh''-ch'+D\F_{k_1,k_2,r}(\Phi(\cdot))^Th. 
\end{equation}
We observe that a differentiation of \eqref{compactform} gives $\L (\Phi')=0$.

The proof of the following result follow closely the arguments contained in  \cite[Lemma 2.2]{kenne2026}. We omit it for brevity. 

\begin{lemma}\label{lemdecay} Let $\pi_0=(k_{1,0}, k_{2,0},r_0,d_0) \in Q$ and let $(\Phi_0, c_0)$ be its monotone front. There exist $C,\eta>0$ such that
	\begin{equation}\label{cy0}
		\|\Phi_0(\xi)-E_-\| + \|\Phi_0'(\xi)\|  \leq C e^{\eta\xi}, \qquad \xi\leq 0
	\end{equation}
	and 
	\begin{equation}\label{cy1}
		\|\Phi_0(\xi)-E_+\| + \|\Phi_0'(\xi)\|  \leq C e^{-\eta\xi}, \qquad \xi\geq 0.
	\end{equation}
	In particular,
	\begin{equation}\label{cy2}
		\Phi'_0\in H^2(\R)^2, \qquad \Phi_0'' \in L^2(\R)^2, \qquad \F_{k_{1,0}, k_{2,0},r_0}(\Phi_0)\in L^2(\R)^2.
	\end{equation}
Moreover, for every integer $m\geq 1$, there exists $C_m>0$ such that
	\begin{equation}\label{generaldecay}
		\|\Phi_0^{(m)}(\xi)\|\leq C_me^{-\eta|\xi|}, \qquad \text{ for all } \xi\in \R.
	\end{equation}
\end{lemma}

 We have the following result. 

\begin{proposition}\label{propfredholm}
Let $k_1,k_2>1$ and $r,d>0$. Then the operator $\L$ is Fredholm of index $0$.
Moreover,
\begin{equation}\label{kerL}
\operatorname{ker}\L=\operatorname{span}\left\{\Phi'\right\}, \qquad\operatorname{codim} \operatorname{Ran} \mathcal{L}=1
\end{equation}
and there exists  $\Psi:=(-a,b)^T\in H^2(\R)^2$ with $a,b>0$ and $(a,b)(\pm \infty)=0$ such that
\begin{equation}\label{kerL*}
	\operatorname{ker}\L^*=\operatorname{span}\left\{ \Psi\right\}.
\end{equation}

\end{proposition}

\begin{proof}
The Fredholm property is sketched in \cite{kenne2026}, so we omit it here.
The first equality in \eqref{kerL} is contained in  \cite[Lemma 3.3]{kan1996} for the case of the operator $\L$ defined on the space $X=\BUC$. Here, the operator $\L: H^2(\R)^2 \to L^2(\R)^2$ and we can adapt their results (see also \cite{risler2017}).  However, we will provide a complete proof for clarity and for further needs. 
Let $h=(h_1,h_2)^T\in H^2(\R)^2$. Then $\L h=0$ is equivalent to 
\begin{equation}\label{syst1}
\left\{ 
	\begin{array}{llllll}
	h_1''+ch_1'+(1-2U-k_1V)h_1-k_1Uh_2=0,\\
	dh_2''+ch_2'+r(1-2V-k_2U)h_2-rk_2Vh_1=0.
	\end{array} 
	\right.
\end{equation}
We transform the system \eqref{syst1} to a cooperative system by letting $w_1=-h_1$ and $w_2=h_2$. We thus obtain 
\begin{equation}\label{syst2}
	\left\{ 
	\begin{array}{llllll}
		w_1''+cw_1'+(1-2U-k_1V)w_1+k_1Uw_2=0,\\
		dw_2''+cw_2'+r(1-2V-k_2U)w_2+rk_2Vw_1=0.
	\end{array} 
	\right.
\end{equation}
Note that because $\L(\Phi')=0$, we have from \eqref{eq2a} that $w=(-U',V')^T$ is a positive solution to \eqref{syst2}. In addition, we obtain from Lemma \ref{lemdecay} that $\dis \lim_{\xi\to \pm\infty}w(\xi)=0$. Next, we can write the system \eqref{syst2} on the form
\begin{equation}\label{compact00}
\tilde{\L}w:=\D_dw''+\C w'+\B(\xi)w=0.
\end{equation} 
where $\D_d$ is defined in \eqref{compactform1}, $\C= \begin{pmatrix}c&0\\0&c\end{pmatrix}$ and $\B(\xi)=\begin{pmatrix}
	1-2U(\xi)-k_1V(\xi) & k_1U(\xi) \\
	rk_2V(\xi)& r(1-2V(\xi)-k_2U(\xi))
\end{pmatrix}.$ 
Then the limiting matrices $\dis \B_{\pm}:=\lim_{\xi\to \pm\infty}\B(\xi)$ have each two negative eigenvalues, so they have negative principal eigenvalues.  Finally, $\B(\xi)$ is irreductible in the functional sense because $k_1U(\xi)>0$ and $rk_2V(\xi)>0$. We apply \cite[Chapter 4, Theorem 5.1]{volpert1994} to deduce that $\operatorname{ker} \tilde{\L}=\operatorname{span}\{w\}$. Therefore, $\operatorname{ker}\L=\operatorname{span}\left\{\Phi'\right\}.$ Moreover, the adjoint equation $\tilde{\L}^*\tilde{w}=0$, with $\tilde{w}(\pm \infty)=0$ has a unique positive solution up to a constant factor. Hence, there exists $a,b>0$ with  $(a,b)(\pm \infty)=0$,  such that $\operatorname{ker}\tilde{\L}^*=\operatorname{span}\left\{ \begin{pmatrix} a\\b\end{pmatrix} \right\}$.  This implies that $\operatorname{ker}\L^*=\operatorname{span}\left\{ \begin{pmatrix} -a\\b\end{pmatrix} \right\}$. The exponential dichotomies of the limiting adjoint systems gives exponential decay of this solution and of its first derivative. Using the adjoint differential equation, we can deduce the exponential decay of its second derivative. Hence, $(a,b)\in H^2(\R)^2$. 
\end{proof}

\begin{remark}\label{remark1} The following observations are in order.
We note that since $\L$ is Fredholm, $\operatorname{Ran} \L$ is closed. Then we can write the orthogonal decomposition $L^2(\mathbb{R})^2=\operatorname{Ran} \L \oplus(\operatorname{Ran}\L)^{\perp}$. Next, as $\operatorname{ker} \L^*=(\operatorname{Ran} \L)^{\perp}$, we can deduce from  \eqref{kerL*} that 
\begin{equation}\label{ranL}
\operatorname{Ran}\L =\left\{f\in L^2(\R)^2: \langle f,\Psi\rangle_{L^2}=0\right\}, \quad \Psi=(-a,b)^T, \quad a,b>0.
\end{equation}
Moreover, 
\begin{equation}\label{imp0}
\langle \Psi, \Phi'\rangle_{L^2}=\int_{\R}(-a(\xi)U'(\xi)+b(\xi)V'(\xi))d\xi>0. 
\end{equation}
This implies that $\Phi'\notin \operatorname{Ran}\L$.  

We also mention that a similar result for the adjoint operator is obtained in \cite[Lemma A. 2.]{kan1993} when the speed $c=0$. We believe that one could also adapt their proof for $c\neq0$. 
\end{remark}

Next, we state the following smoothness result. 

\begin{theorem}\label{theosmooth}
	Let $\pi \in Q$ and let $c_{\pi}$ be the speed of the unique monotone front
	$\Phi_{\pi}=(U_{\pi},V_{\pi})^T$ of \eqref{eq1}-\eqref{eq2} associated with
	$\pi$ and normalized by $U_{\pi}(0)=\frac{1}{2}$. Then the speed map
	\begin{equation*}
	\pi \longmapsto c_{\pi}
	\end{equation*}
	is of class $C^\infty$ on $Q$. Moreover, for every
	$\pi_0\in Q$, the map
	\begin{equation*}
		\pi\longmapsto
		\Phi_{\pi}-\Phi_{\pi_0}\in H^2(\R)^2
	\end{equation*}
	is of class $C^\infty$ on $Q$. 
\end{theorem}

\begin{proof} The proof of the local differentiability follows the same arguments as in \cite[Theorem 3.2]{kenne2026}. We only upgrade from the local differentiability to the global one. For every fixed $\pi_0\in Q$, there exists a neighbourhood $\V_0$ of $\pi_0\in Q$ such that the map
	\begin{equation*}
		\pi \longmapsto
		\Phi_{\pi}-\Phi_{\pi_0}\in H^2(\R)^2
	\end{equation*}
	is of class $C^\infty$ on $\V_0$.  Now let $\pi_1\in Q$ be arbitrary. Then there exists a neighbourhood $\V_1$ of $\pi_1$ such that the map
	\begin{equation*}
		\pi\mapsto
		\Phi_{\pi}-\Phi_{\pi_1}\in H^2(\R)^2
	\end{equation*}
	is of class $C^\infty$ on $\V_1$. Let $\pi\in \V_1$. We write
	\begin{equation}\label{decomposition}
		\Phi_{\pi}-\Phi_{\pi_0}= (\Phi_{\pi}-\Phi_{\pi_1}) + (\Phi_{\pi_1}-\Phi_{\pi_0}).
	\end{equation}
	The first term of \eqref{decomposition} belongs to $H^2(\R)^2$. We can deduce from Lemma \ref{lemdecay} that the second term belongs also to $H^2(\R)^2$.
	Since $\pi_1\in Q$ was arbitrary, we conclude that the  map $\pi\mapsto \Phi_{\pi}-\Phi_{\pi_0}$ is of class $C^\infty$ from $Q$ into $H^2(\R)^2$. 
	This completes the proof.
\end{proof}

We now state the following result determining the sign of the wave speed in the symmetric case. This result is obtained from the particular case $r=1$ in \cite[Theorem 1.1]{kenne2026} by using a scaling argument (see e.g., \cite{guo2013}).

 \begin{lemma}\label{lemmasign}
	Let $k_1=k_2:=k>1$ and $r,d>0$. Then
	\begin{equation}\label{eq3}
		\sign c_{k,k,r,d}=\sign(r-d).
	\end{equation}
That is
	\begin{equation}\label{eq4}
		c_{k,k,r,d}\left\{ 
		\begin{array}{llllll}
			>0, &  0<d<r,\\
			=0, &  d=r,\\
			<0, & d>r.
		\end{array} 
		\right.
	\end{equation}
\end{lemma}

\begin{proof}
First, it is worth mentioning that the orientation of profiles in \cite{guo2013} are opposite to ours, so they speed $s=-c$. Next from \cite[Theorem 1.3]{guo2013}, we have that for any $l>0$, $\sign s(a,k,k,d)=\sign s(la,k,k,ld)$. In our setting, this means $\sign c_{k,k,r,d}=\sign c_{k,k,lr,ld}$ for any $l>0$. In particular for $l=\frac{1}{r}$, we have $\sign c_{k,k,r,d}=\sign c_{k,k,1,\frac{d}{r}}$.  From \cite[Theorem 1.1]{kenne2026},  $\sign c_{k,k,1,\delta}=\sign (1-\delta)$ for all $\delta>0$. Hence, by choosing $\delta=\frac{d}{r}>0$, we obtain
\begin{equation*}
\sign c_{k,k,r,d} =\sign c_{k,k,1,\frac{d}{r}}=\sign \left(1-\frac{d}{r}\right)=\sign(r-d),
\end{equation*}
because $r>0$. 
\end{proof}

\begin{remark}\label{rem1} The following observations are in order.
Here, $c_{k,k,r,d}<0$ means that the territory occupied by the species $v$ expands to the left, while $c_{k,k,r,d}>0$ means that the species $u$ advances to the right. Since $d=\frac{d_2}{d_1}$ and  $r=\frac{r_2}{r_1}$, we can rewrite \eqref{eq4} as 
\begin{equation}
c_{k,k,r,d} \left\{ 
	\begin{array}{llllll}
		>0, &  0<\frac{d_2}{r_2}<\frac{d_1}{r_1},\\
		=0, &  \frac{d_2}{r_2}=\frac{d_1}{r_1},\\
		<0, & \frac{d_2}{r_2}>\frac{d_1}{r_1}.
	\end{array} 
	\right.
\end{equation}
When $r_1=r_2$, we obtain the ``Unity is not strength" result (see e.g., \cite{girardin2019,girardin2018}): In a homogeneous environment under strong competition, the faster diffuser advances.  However, when the species have different growth rates, the outcome may change. For example we can have $d_2>d_1$ but the species $v$ recedes. This happens when its intrinsic growth rate is sufficiently larger compared to the one of the species $u$. Conversely, the species $v$ may advances  if $d_2<d_1$ if its intrinsic growth rate is sufficiently smaller compared to the one of the species $u$. 
\end{remark}

Next, for standing fronts, we use the normalization $(r,d)=(1,\delta)$ where $\delta=\frac{d}{r}$. The profile equations \eqref{eq1}-\eqref{eq2} with $c=0$ become
\begin{equation}\label{standing0}
U''+U(1-U-k_1V)=0, \qquad \delta V''+V(1-V-k_2U)=0.
\end{equation}

\begin{lemma}\label{lemineq1}
let $k_1,k_2>1$ and  let $(U,V)$ be a monotone standing front. We set $W:=1-U-V$. Then $W>0$ in $\R$ and
	\begin{equation}\label{ari2}
		0<-U'(\xi)\leq\sqrt{k_1-1}\,U(\xi),\qquad
		0<V'(\xi)\leq\sqrt{\frac{k_2-1}{\delta}}\,V(\xi),\qquad \xi\in\R.
	\end{equation}
\begin{equation}\label{ari3}
	|U''(\xi)|\leq k_1U(\xi),\qquad
	|V''(\xi)|\leq\frac{k_2}{\delta}V(\xi),\qquad \text{for all } \xi\in\R,
\end{equation}
and
\begin{equation}\label{ari4}
	\lim_{\xi\to+\infty}\frac{(U'(\xi))^2}{U(\xi)}=0,
	\qquad
	\lim_{\xi\to-\infty}\frac{(V'(\xi))^2}{V(\xi)}=0.
\end{equation}
Finally, the following identities hold true:
	\begin{equation}\label{ari0}
	\int_{\R} UVU'd\xi=-\frac1{6k_1},\qquad 	\int_{\R} UVV'd\xi=\frac1{6k_2},
	\end{equation}
	\begin{equation}\label{ari1}
		\int_{\R} V^2U'd\xi=-\frac1{3k_2},\qquad \int_{\R} U^2V'd\xi=\frac1{3k_1}.
	\end{equation}
\end{lemma}

\begin{proof}
	The proof of $W>0$ and  \eqref{ari0}-\eqref{ari1} follows from similar arguments as in \cite[proof of Theorem 4.1 and Corollary 4.2]{kenne2026} (see also \cite{guo2013}).   We only prove the first identity in \eqref{ari2}, the second is obtained by similar arguments. 
	We rewrite \eqref{standing0} as
	\begin{equation}\label{standing01}
		U''=U((k_1-1)V-W)=0, \qquad \delta V''=\frac{V}{\delta}((k_2-1)U-W)=0.
	\end{equation}
	Since $0<U,V<1$ and $W>0$, it follows from \eqref{standing01} that 
		\begin{equation}\label{standing02}
		U''< U(k_1-1), \qquad \delta V''< \frac{(k_2-1)}{\delta}V.
	\end{equation}
	Now set $P=(U')^2-(k_1-1)U^2$ then since $U'<0$, we have thanks to \eqref{standing02}  that $P'=2(U'(U''-(k_1-1)U))>0$. Using Lemma \ref{lemdecay}, we deduce that $\dis \lim_{\xi\to +\infty}P(\xi)=0$. Combining the previous facts imply that $P(\xi)<0$ for all $\xi \in \R$. That is $(U')^2<(k_1-1)U^2$ on $\R$. But $U'<0$, so this latter inequality gives the first identity in \eqref{ari2}. Similar arguments with $Q:=(V')^2-\frac{k_2-1}{\delta} V^2$ at $-\infty$ gives the second identity in \eqref{ari2}.  
	Next, from \eqref{standing0}, we use again $0<U,V<1$ and since $k_1,k_2>1$ to deduce that
\begin{equation*}
	|U''|=U|1-U-k_1V|\leq k_1U, \qquad |V''|=\frac{V}{\delta}|1-V-k_2 U| \leq \frac{k_2}{\delta} V.
\end{equation*}
Hence \eqref{ari3} holds. Finally taking the square in  both inequality in \eqref{ari2} while using \eqref{eq2} gives  \eqref{ari4}. This completes the proof.
\end{proof}

\begin{proposition}\label{propenergyidentities} Let $(U,V)$ be a solution to \eqref{standing0}. Then we have
	\begin{equation}\label{energyimp1}
	\frac{k_1^2}{6k_2}-\frac{1}{6}=\frac{1}{2}\int_{\R}\frac{(-U')^3}{U^2}d\xi+\frac{1}{2}\int_{\R}\frac{(-U')(U'')^2}{U^2}d\xi
	\end{equation}
	and 
	\begin{equation}\label{energyimp2}
	\frac{k_2^2}{6k_1}-\frac{1}{6}=\frac{\delta}{2}\int_{\R}\frac{(V')^3}{V^2}d\xi+\frac{\delta^2}{2}\int_{\R}\frac{V'(V'')^2}{V^2}d\xi.
	\end{equation}
	The four integrals are finite and the two right-hand sides are strictly positive.
\end{proposition}

\begin{proof}
	We only prove \eqref{energyimp1}, \eqref{energyimp2} being obtained similarly. We set
	\begin{equation}\label{ari5}
		S:=1-U-k_1V=-\frac{U''}{U},
	\end{equation}
	where the second equality follows from the first equation in \eqref{standing0}. Then
	$k_1V=1-U-S$ and $k_1V'=-U'-S'$. Next, we multiply the second identity of \eqref{ari0}, by $k_1^2$ to obtain
	\begin{equation}\label{ari6}
		\frac{k_1^2}{6k_2}=\int_{\R} U(1-U-S)(-U'-S')d\xi=\int_{\mathbb{R}} U(1-U)\left(-U^{\prime}\right) d \xi-\int_{\mathbb{R}} U(1-U) S^{\prime} d \xi -\int_{\mathbb{R}} U S\left(-U^{\prime}\right) d \xi+\int_{\mathbb{R}} U S S^{\prime} d \xi
	\end{equation}
	With a change of variables, the first integral in \eqref{ari6} gives $\int_{\R} U(1-U)(-U')d\xi =\frac{1}{6}$.   Using an integration by parts, the remaining three integrals give
	\begin{equation}\label{ari7}
		-\int_\R(1-U)(-U')Sd\xi+\frac12\int_\R(-U')S^2d\xi.
	\end{equation}
	The boundary terms are all zero because $(U,V,S)(-\infty)=(1,0,0)$ and
	$(U,V,S)(+\infty)=(0,1,1-k_1)$. Using the expression of $S$ given in \eqref{ari5} in the first term of
	\eqref{ari7} and  thanks again to an integration by parts, we obtain 
	\begin{equation*}
		\begin{array}{lllll}
\dis -\int_{\R}(1-U)(-U')Sd\xi= \int_{\R}\frac{1-U}{U}(-U')U''d\xi=\frac{1}{2}\int_{\R}\frac{(-U')^3}{U^2}d\xi.
		\end{array}
	\end{equation*}
	The boundary term vanishes at $-\infty$ because $1-U\to0$ and $U'\to 0$, and at $+\infty$
	because $\frac{(U')^2}{U}\to 0$ (by Lemma \ref{lemineq1}). The second term of
	\eqref{ari7} is $\frac{1}{2}\int_\R(-U')\frac{(U'')^2}{U^2}d\xi$ thanks to \eqref{ari5}. This proves
	\eqref{energyimp1}. For \eqref{energyimp2}, we set  $R:=1-V-k_2U=-\delta\frac{V''}{V}$ and we perform similar arguments as above. 

	Finally, using Lemmas \ref{lemdecay} and \ref{lemineq1}, we can deduce that the four integrals are finite. Since $U'<0<V'$, 
	the first integral in \eqref{energyimp1} and \eqref{energyimp2} is strictly positive. This completes the proof.
\end{proof}
\begin{corollary}\label{corobounds}
	Let $k_1,k_2>1$. Assume that a standing monotone front exists for $(k_1,k_2,\delta)$, then
	\begin{equation}\label{strictbounds}
		\sqrt{k_2}<k_1<k_2^2.
	\end{equation}
\end{corollary}

\begin{proof}
	The right-hand side of \eqref{energyimp1} is strictly positive, so $k_1^2>k_2$. Also, the right-hand side of
	\eqref{energyimp2} is strictly positive, so $k_2^2>k_1$.
\end{proof}

The following result gives another important energy identity. It is exactly the same identity  in \cite[Theorem 1.1]{ma2026} after changing the orientation. It follows by the same integrations by parts used in \cite[Theorem 4.1]{kenne2026}, while keeping $k_1$ and $k_2$ independent.
\begin{proposition}\label{zerospeed}
	Let $(U,V)$ be a monotone standing front and set $W=1-U-V$. Then
	\begin{equation}\label{imp1}
	J:=\int_{\R} W^2(-U')d\xi=\int_{\R} W^2V'd\xi>0
	\end{equation}
	and 
		\begin{equation}\label{imp2}
	(1-\delta)J =\frac{(k_1-k_2)[(k_2-1)+\delta(k_1-1)]}{3k_1k_2}.
	\end{equation}
\end{proposition}

\begin{proof}
	Since $W'=-U'-V'$ and $W(\pm\infty)=0$, we have
\begin{equation*}
\int_{\R} W^2(-U')d\xi-\int_{\R} W^2V'd\xi=\int_{\R} W^2W'd\xi=0
\end{equation*}
which proves \eqref{imp1}. Since $(U'V')(\pm\infty) =0$ then
	$\int_{\R} U''V'd\xi=-\int_{\R} U'V''d\xi$. We multiply the first and the second equations in \eqref{standing0} by $V'$ and $U'$, respectively, we integrate over $\R$ and we combine the resulting equalities to obtain
	\begin{equation}\label{ari8}
		\delta\int_{\R} U(1-U-k_1V)V'd\xi+\int_{\R} V(1-V-k_2U)U'd\xi=0.
	\end{equation}
	Moreover,
	\begin{equation*}
		2U(1-U-k_1V)+W^2=(1-V)^2-U^2-2(k_1-1)UV
	\end{equation*}
	and 
\begin{equation*}
2V(1-V-k_2U)+W^2=(1-U)^2-V^2-2(k_2-1)UV.
\end{equation*}
Thus multiplying the first equality by $V'$ and the second one by $U'$ and integrating over $\R$, while using \eqref{ari0}-\eqref{ari1}, $\int_{\R}(1-V)^2V'd\xi=\frac{1}{3}$ and $\int_{\R}(1-U)^2U'd\xi=-\frac{1}{3}$ we arrive at
\begin{equation*}
2\int_{\R} U(1-U-k_1V)V'd\xi+J =-\frac{(k_1-1)(k_1-k_2)}{3k_1k_2}
\end{equation*}
and 
\begin{equation*}
2\int_{\R} V(1-V-k_2U)U'd\xi-J=-\frac{(k_2-1)(k_1-k_2)}{3k_1k_2},
\end{equation*}
respectively.	Substituting these relations  in \eqref{ari8} gives \eqref{imp2}.
\end{proof}

\begin{corollary}\label{corollarysign}
	Let $k_1,k_2>1$ and let $(U,V)$ be a monotone standing front. Then
	\begin{equation}\label{signrel}
		\sign(k_1-k_2)=\sign(1-\delta),
	\end{equation}
	with the convention that the two sides vanish simultaneously.
\end{corollary}

\begin{proof}
	In \eqref{imp2}, the terms $J$, $3k_1k_2$ and $(k_2-1)+\delta(k_1-1)$ are strictly positive. The conclusion follows.
\end{proof}
\begin{remark}\label{remarksiha}
We notice that when $k_1=k_2$, we obtain from \eqref{signrel} that $\delta=1$. That is $d=r$. So a standing front can only exists when $\delta=r$. This is a generalization of \cite[Theorem 4.1]{kenne2026}. 
\end{remark}

The next result constructs the unique zero-speed threshold and completely characterizes the sign of the wave speed. A similar result was established in \cite[Theorem 6.1]{nakamura2026}. We also provide explicit bounds and additional properties.  We use Corollary \ref{corobounds}, the monotonicity of the speed with respect to the reaction parameters established in \cite{kan1995} and the continuity of the front.
\begin{proposition}\label{propthreshold}
	For every $k_2>1$ and $\delta>0$ there exists a unique number $\K(\delta;k_2)$ such that, for
	every $d,r>0$ with $\delta=\frac{d}{r}$ and every $k_1>1$,
	\begin{equation}\label{signK00}
		c_{k_1,k_2,r,d}\begin{cases}
			>0,&k_1<\K(\delta;k_2),\\
			=0,&k_1=\K(\delta;k_2),\\
			<0,&k_1>\K(\delta;k_2).
		\end{cases}
	\end{equation}
	It satisfies $\sqrt{k_2}<\K(\delta;k_2)<k_2^2$ and
	$\mathcal K(1;k_2)=k_2$. Moreover,
	\begin{equation}\label{Kexchange}
		\alpha=\K(\delta;\beta)\quad\Longleftrightarrow\quad
		\beta=\K(\delta^{-1};\alpha)
	\end{equation}
	for every $\alpha,\beta>1$ and $\delta>0$.
\end{proposition}

\begin{proof} Let $k>1$. When $k_1=k_2=k$ and $d=r=1$ we have (see e.g., \cite[Theorem 1.1]{guo2013}) that
	$c_{k,k,1,1}=0$. Since the speed is strictly decreasing with respect to $k_1$ \cite{kan1995}, we obtain 
	\begin{equation}\label{diagsign}
		c_{k_1,k_2,1,1}>0\quad\text{ if }1<k_1<k_2  \quad \text{ and }
		\quad c_{k_1,k_2,1,1}<0\quad\text{ if }k_1>k_2.
	\end{equation}
	Next, fix $k_2>1$ and let $r=1$, $d=\delta$. Since $1<\sqrt{k_2}<k_2<k_2^2$,  we can deduce from \eqref{diagsign} that when $\delta=1$, the speed is positive at $k_1=\sqrt{k_2}$ and  is negative at $k_1=k_2^2$. We claim that for all $\delta>0$, 
\begin{equation*}
	c_{\sqrt{k_2},k_2,1,\delta}>0  \quad \text{ and }
	\quad c_{k_2^2,k_2,1,\delta}<0.
\end{equation*}
Indeed, assume that there exists some $\delta>0$ so that $c_{\sqrt{k_2},k_2,1,\delta}=0$. Then there is a standing front for $k_1=\sqrt{k}_2$, which contradicts \eqref{strictbounds}. Therefore,  $c_{\sqrt{k_2},k_2,1,\delta}\neq 0$ for all $\delta>0$. Using Theorem \ref{theosmooth}, the map $\delta \mapsto c_{\sqrt{k_2},k_2,1,\delta}$ is continuous on the connected set $(0,\infty)$ and since $\delta \mapsto c_{\sqrt{k_2},k_2,1,1}>0$, we conclude that the first inequality in the claim holds. With similar arguments, we can deduce the second inequality in the claim. Thanks again to the continuity and the strict monotonicity of the speed with respect to $k_1$, we can deduce  that the map $k_1\mapsto c_{k_1,k_2,1,\delta}$ has a unique zero $\K(\delta;k_2)\in(\sqrt{k_2},k_2^2)$. This proves \eqref{signK00} with $(r,d)=(1,\delta)$. Now, from \eqref{standing0}, we have $\c=0  \Longleftrightarrow c_{k_1, k_2, 1, \delta}=0$. Hence, $\c=0  \Longleftrightarrow k_1=\K(\delta ; k_2)$. Using again the fact that the map $k_1\mapsto \c$ is strictly decreasing imply that  $\c>0$ when $k_1<\K(\delta ; k_2)$ and $\c<0$ when  $k_1>\K(\delta ; k_2)$. Hence \eqref{signK00} holds. 
Next, when $\delta=1$, we already know that $c_{k_2, k_2, 1,1}=0$. By uniqueness of the zero we obtain $\K(1 ; k_2)=k_2$. Finally, we recall the following exchange relation (see e.g., \cite{ ma2019, nakamura2026}. A standing front exists for $(k_1,k_2,\delta)=(\alpha,\beta,\delta)$ if and only if  a standing front exists for $(\beta, \alpha,\delta^{-1})$. The profiles are related by the change of variables $\tilde{U}(\eta)=V(-\sqrt{\delta}\eta)$ and $\tilde{V}(\eta)=U(-\sqrt{\delta}\eta)$.   Thus using the uniqueness of the threshold, $\alpha=\K(\delta;\beta) \Longleftrightarrow c_{\alpha, \beta, 1, \delta}=0\Longleftrightarrow c_{\beta,\alpha, 1, \delta^{-1}}=0 \Longleftrightarrow \beta=\K(\delta^{-1};\alpha)$. Hence \eqref{Kexchange} holds. This completes the proof.
\end{proof}

\section{Monotonicity of the speed with respect to the diffusion ratio $d$}\label{monotonicity}
The aim of this section is to study the monotonicity of the wave speed with respect to the diffusion ratio $d$ and so to prove the Theorem \ref{maintheorem}. We also prove Corollary \ref{corollary1}, Theorem \ref{maintheorem1} and finally deduce Theorem \ref{maintheorem2} and Corollary \ref{corollarymonotonegeneral}.

Let us fix $k_1,k_2>1$ and $r>0$ and define $\Phi:=\Phi_{k_1,k_2,r,d}=(U_{k_1,k_2,r,d},V_{k_1,k_2,r,d})^T$ and  $c:=\c$. Let $\pi_0:=(k_{1,0},k_{2,0}, r_0,d_0)\in Q$ be fixed and consider the profile $\Phi_0:=\Phi_{k_{1,0},k_{2,0},r_0,d_0}$ associated. 

We set $w=\Phi-\Phi_0$. Then from Theorem \ref{theosmooth}, $w\in C^\infty((0,\infty);H^2(\R)^2)$ and since $\Phi_0$ does not depends on $d$, we define $\partial_d\Phi:=\partial_dw\in H^2(\R)^2$.  Moreover, \eqref{compactform} can be rewritten as
\begin{equation}\label{newcomp}
	\D_d (\Phi_0+w)''+c(d) (\Phi_0+w)'+\F_{k_1,k_2,r}(\Phi_0+w)=0.
\end{equation}
Differentiating \eqref{newcomp} with respect to $d$ leads us to
\begin{equation*}
	\left(\begin{array}{ll}0 & 0 \\ 0 & 1\end{array}\right)\Phi''+\D_d  (\partial_d\Phi)''+\partial_dc \Phi'+c(d)(\partial_d\Phi)'+D\F_{k_1,k_2,r}(\Phi)\partial_d \Phi=0.
\end{equation*}
That is
\begin{equation}\label{eq0}
	\L(\partial_d \Phi)+\partial_dc \,\Phi'+\begin{pmatrix} 0\\V''\end{pmatrix}=0,
\end{equation}
where $V:=V_{k_1,k_2,r,d}$. Recall that $\operatorname{ker} \L^*=(\operatorname{Ran} \L)^{\perp}$ and let $\Psi\in \operatorname{ker} \L^*$. Then, from Proposition \ref{propfredholm} we have $\Psi=(-a,b)^T$ with $a,b>0$. Thanks to Remark \ref{remark1}, we have $\langle \Psi, \Phi'\rangle_{L^2}>0$. Therefore, by taking the $L^2$-scalar product with $\Psi$ in \eqref{eq0}, we obtain
\begin{equation}\label{eq01}
	\left \langle\L(\partial_d \Phi), \Psi\right\rangle_{L^2}+\partial_dc \left \langle \Phi', \Psi\right\rangle_{L^2}+\left \langle\begin{pmatrix} 0\\V''\end{pmatrix},\Psi\right\rangle_{L^2}=0.
\end{equation}
The first term in \eqref{eq01} is equal to $0$ and since $\left \langle \Phi', \Psi\right\rangle_{L^2}\neq 0$, we deduce that
\begin{equation*}
	\partial_dc =-\frac{\left \langle\begin{pmatrix} 0\\V''\end{pmatrix},\Psi\right\rangle_{L^2}}{\left \langle \Phi', \Psi\right\rangle_{L^2}}.
\end{equation*}
That is 
\begin{equation}\label{eq02}
\dis 	\partial_dc =-\frac{\dis \int_{\R}V''(\xi)b(\xi)\,d\xi}{\left \langle \Phi', \Psi\right\rangle_{L^2}}.
\end{equation}

Now let us recall that $\L \Phi'=0$. So, if we set $p=-U'>0$ and $q=V'>0$ we obtain that $(p,q)$ solves the system
\begin{equation}\label{syst2new}
	\left\{ 
	\begin{array}{llllll}
		p''+cp'+(1-2U-k_1V)p+k_1Uq=0,\\
		dq''+cq'+r(1-2V-k_2U)q+rk_2Vp=0.
	\end{array} 
	\right.
\end{equation}
Moreover, $(a,b)$ obtained in Proposition \ref{propfredholm} solves
\begin{equation}\label{systab}
	\left\{ 
	\begin{array}{llllll}
		a''-ca'+(1-2U-k_1V)a+rk_2Vb=0,\\
		db''-cb'+r(1-2V-k_2U)b+k_1Ua=0,
	\end{array} 
	\right.
\end{equation}
and 
\begin{equation}\label{limab}
(a,b)(\pm \infty)=0.
\end{equation}
We can rewrite \eqref{eq02} as follows:
\begin{equation}\label{eq03}
	\dis 	\partial_dc =-\frac{\dis \int_{\R}q'(\xi)b(\xi)\,d\xi}{\left \langle \Phi', \Psi\right\rangle_{L^2}}.
\end{equation}

\begin{theorem}\label{theoimportant}
	Let $k_1,k_2>1$, $r>0$ and  $c:=\c$. We assume that 
	\begin{equation}\label{condition}
	c\leq 0	 \; \text{ and }\; c\left(1-\frac{1}{d}\right)\leq 0
	\end{equation}
hold. Then
	\begin{equation}\label{eq04}
		K:=\int_{\R}q'(\xi)b(\xi)d\xi>0.
	\end{equation}
\end{theorem}
\begin{proof}
We first observe that an integration by parts over $\R$ while using Lemma \ref{lemdecay} and \eqref{limab} gives
\begin{equation}\label{important0}
	2K= \int_{\R}\left(q'(\xi)b(\xi)-q(\xi)b'(\xi) \right)d\xi. 
\end{equation}
Now, we multiply the first equation in \eqref{syst2new} by $a$ and the first equation in \eqref{systab} by $p$ and we subtract the resulting equations to obtain
\begin{equation*}
a''p-p''a-c(p'a+a'p)=k_1Uqa-rk_2Vbp.
\end{equation*}
since $(a'p-p'a)'=a''p-p''a$, we can rewrite the previous equation as 
\begin{equation}\label{siha0}
	(a'p-p'a-cap)'=k_1Uqa-rk_2Vbp.
\end{equation}
Next, multiplying the second equation in \eqref{syst2new} by $b$ and the second equation in \eqref{systab} by $q$ and combining the resulting equations, we obtain
\begin{equation}\label{siha1}
	[d(q'b-b'q)+cbq]'=k_1Uqa-rk_2Vbp.
\end{equation}
On the one hand, combining \eqref{siha0} and \eqref{siha1} leads us to 
\begin{equation*}
	[d(q'b-b'q)+cbq]'=(a'p-p'a-cap)'.
\end{equation*}
So, the functions $\xi \mapsto d(q'(\xi)b(\xi)-b'(\xi)q(\xi))+cb(\xi)q(\xi)$ and $\xi \mapsto a'(\xi)p(\xi)-p'(\xi)a(\xi)-ca(\xi)p(\xi)$ have the same derivatives. Using Lemma \ref{lemdecay} and thanks to the fact that $\xi \mapsto a'(\xi)$ and $\xi\mapsto b'(\xi)$ are bounded (because $H^1(\R)\hookrightarrow  C_b(\R)$), we can deduce that these two functions have the same limit $0$ at infinity. Therefore, these two functions must be equal. We set for all $\xi \in \R$,
\begin{equation}\label{defG}
G(\xi):=d(q'(\xi)b(\xi)-b'(\xi)q(\xi))+cb(\xi)q(\xi)=a'(\xi)p(\xi)-p'(\xi)a(\xi)-ca(\xi)p(\xi).
\end{equation} 
Then we have 
\begin{equation}\label{limderivG}
\lim_{\xi\to \pm\infty}G(\xi)=0  \;  \text{ and } \; G'(\xi)= k_1U(\xi)q(\xi)a(\xi)-rk_2V(\xi)b(\xi)p(\xi).
\end{equation}
Differentiating the equation in \eqref{limderivG}, gives
\begin{equation}\label{siha2}
G''=k_1\left[-pqa+ Uq'a+Uqa'\right] -rk_2\left[qbp+Vb'p+Vbp'\right],
\end{equation}
where we dropped the variable $\xi$ for clarity.
From \eqref{defG}, we can write $a'=\frac{G}{p}+\frac{p'}{p}a+ca$ and $q'=\frac{G}{d b}+\frac{b'}{b}q-\frac{c}{d}q$. Substituting these values in \eqref{siha2} lead us to
\begin{equation*}
G''=\left[ \left( \frac{b'}{b}+\frac{p'}{p}\right)G' +k_1U\left( \frac{a}{d b}+\frac{q}{p}\right) G+ k_1c\left(1-\frac{1}{d}\right)Uqa -pq(k_1a+rk_2b)\right].
\end{equation*}
That is 
\begin{equation}\label{siha3}
\M G:=G''-\left( \frac{b'}{b}+\frac{p'}{p}\right)G' -k_1U\left( \frac{a}{d b}+\frac{q}{p}\right) G=q\left[ k_1c\left(1-\frac{1}{d}\right)Ua -p(k_1a+rk_2b)\right].
\end{equation}
Since $k_1,k_2>1, q,U,a,p,r,b>0$,  if $c\left(1-\frac{1}{d}\right)\leq0$, then the right hand side of \eqref{siha3} is strictly negative. 
That is  
\begin{equation}\label{siha4}
\M G(\xi)<0\; \text{ for all }\;  \xi\in \R.
\end{equation}
Note that the same condition  and \eqref{siha3} imply that $G$ is not identically $0$. Now, we claim that $G>0$ on $\R$. Indeed, suppose that there is $\xi_0\in \R$ so that $G(\xi_0)\leq 0$.  Then because  $\dis \lim_{\xi\to \pm\infty}G(\xi)=0$ and $G$ is continuous, $G$ attains a nonpositive minimum at some point $\xi_1 \in \R$. Thus at $\xi_1$, we have 
\begin{equation*}
G(\xi_1)\leq 0, \quad G'(\xi_1)=0 \text{ and } G''(\xi_1)\geq 0.
\end{equation*}
Since $k_1U\left( \frac{a}{d b}+\frac{q}{p}\right)>0$, we obtain from \eqref{siha3} that $\M G(\xi_1)\geq 0$. This contradicts \eqref{siha4}. Hence $G>0$ on $\R$. Now, from \eqref{defG}, we have 
\begin{equation*}
q'(\xi)b(\xi)-b'(\xi)q(\xi)=\frac{1}{d}\left( G(\xi)-cb(\xi)q(\xi)\right),
\end{equation*}
which combined with \eqref{important0} imply 
\begin{equation*}
K= \frac{1}{2d}\int_{\R}\left( G(\xi)-cb(\xi)q(\xi)\right)d\xi. 
\end{equation*}
Combining $G>0$ and $c\leq 0$, we deduce that $K>0$.  This completes the proof.
\end{proof}

We can now prove our main result.
\begin{proof}[Proof of Theorem \ref{maintheorem}]Let $k_1,k_2>1$ and $r>0$. Then from \eqref{eq03}, we have
\begin{equation*}
	\dis 	\partial_dc =-\frac{\dis \int_{\R}q'(\xi)b(\xi)\,d\xi}{\left \langle \Phi', \Psi\right\rangle_{L^2}}.
\end{equation*}
Remark \ref{remark1} implies that the denominator of the right hand side of the previous identity is strictly positive.  Moreover, using Theorem \ref{theoimportant}, we deduce that $K=\int_{\R}q'(\xi)b(\xi)\,d\xi>0$ when $c\leq 0$ and $c\left(1-\frac{1}{d}\right)\leq 0$. Hence $	\partial_dc<0$. This completes the proof. 
\end{proof}

\begin{proof}[Proof of Corollary \ref{corollary1}] For the symmetric case $k_1=k_2=k$, we have from Lemma \ref{lemmasign} that $c\leq 0$ if and only if $d\geq r$.  So the conditions \eqref{condition0} of Theorem \ref{maintheorem} are equivalent to $d\geq r$ and $d\geq 1$ or $d=r$ and $d\leq 1$.  Therefore,  applying Theorem \ref{maintheorem}, we deduce \eqref{eq5a}. 
\end{proof}

\begin{remark}\label{remark2}
We observe from Corollary \ref{corollary1} that we have $\left. \partial_d c_{k,k,r,d} \right|_{d=r}<0$ and from Lemma \ref{lemmasign}, we have $c_{k,k,r,r}=0$. Because $d\mapsto \c$ is smooth on $(0, \infty)$, we can deduce that this function is still strictly decreasing in a neighbourhood of $r$. However, our result does not provide the monotonicity of the function $d\mapsto c_{k,k,r,d}$ on the regions
\begin{equation*}
\begin{cases}
	(0, r-\varepsilon] \cup[r+\varepsilon, 1), & 0<r<1, \\ (0, r-\varepsilon], & r \geq 1,\end{cases}
\end{equation*}
for some $\varepsilon>0$ sufficiently small.  Therefore the global monotonicity remains an open question. 
\end{remark}

\begin{proof}[Proof of Theorem \ref{maintheorem1}]
Let us fix $k_2>1$, $\delta_0>0$ and set $k_{1,0}=\K(\delta_0;k_2)$. We define $F(k_1,\delta):= c_{k_1,k_2,1,\delta}$. Then, thanks to Theorem \ref{theosmooth}, we have that $F$ is of class $C^\infty$ on $(1,\infty)\times (0,\infty)$. Moreover, Proposition \ref{propthreshold} gives the existence of a unique $\K(\delta;k_2) \in (1,\infty)$ so that $F(\K(\delta;k_2),\delta)=0$.  Next, we  let $\Phi_0=(U_0,V_0)^T$ be the standing front at $(k_{1,0},k_2,1,\delta_0)$.  Differentiating \eqref{newcomp} with respect to $k_1$ at $(k_{1,0},\delta_0)$ and using similar arguments as above lead us to
\begin{equation*}
\partial_{k_1}F(k_{1,0},\delta_0)
=-\frac{\dis\int_{\R} aU_0V_0d\xi}
{\left \langle \Phi'_0, \Psi\right\rangle_{L^2}},
\end{equation*}
where $\Psi=(-a,b)^T\in\operatorname{ker}\L^*$ with $a,b>0$. Using Remark \ref{remark1}, we have that the denominator of the previous equality is strictly positive. The numerator is also positive because $a, U_0,V_0>0$. This implies that $\partial_{k_1}F(k_{1,0},\delta_0)<0$. Hence $\partial_{k_1}F(k_{1,0},\delta_0)\neq 0$. We apply the implicit function theorem to deduce  the existence of two intervals $I\subset(0,\infty)$ and $J\subset(1,\infty)$ containing respectively $\delta_0$ and $k_{1,0}$ and a unique function $\kappa\in C^\infty(I;J)$ such that 
\begin{equation*}
\kappa(\delta_0)=k_{1,0} \quad \text{ and } \quad  F(\kappa(\delta),\delta)=0, \quad\text{for every }\delta\in I.
\end{equation*}
The uniqueness of the threshold given in Proposition \ref{propthreshold} implies that  $\kappa(\delta)=\K(\delta;k_2)$ for all $\delta\in I$. So, $\delta \mapsto \K(\delta;k_2)$ is of class $C^\infty$ in the neighbourhood of $\delta_0$. Since $\delta_0$ was arbitrary, it follows that $\delta \mapsto \K(\delta;k_2)$ is of class $C^\infty$ on $(0,\infty)$.

Now, let $\Phi=(U,V)^T$ be the standing front at $(k_{1},k_2,r,d)=(\K(\delta;k_2),k_2,1,\delta)$. Differentiating \eqref{standing0} with respect to $\delta$ gives
\begin{equation}\label{eq05}
	\L_0(\partial_\delta\Phi)-\partial_\delta\K
	\begin{pmatrix}UV\\0\end{pmatrix}
	+\begin{pmatrix}0\\V''\end{pmatrix}=0,
\end{equation}
where $\L_0 h:=\D_{\delta}h''+D \F_{\K,k_2,r}h$. Using Proposition \ref{propfredholm} with $c=0$, we take the $L^2$-scalar product with $\Psi=(-a,b)^T\in\ker\L^*_0$ in \eqref{eq05}, we obtain
\begin{equation*}
\partial_\delta\K\int_{\R} aUVd\xi+\int_\R bV''d\xi=0.
\end{equation*}
That is 
\begin{equation}\label{eq06}
\partial_\delta\K(\delta;k_2) =-\frac{\dis\int_{\R} bV''d\xi} {\dis\int_{\R} aUVd\xi}.
\end{equation}
Since $c=0$, $r=1$ and $d=\delta$, the proof of Theorem \ref{theoimportant} gives $W:=G=a'p-ap'=\delta(bq'-b'q)>0$ and $\dis \int_{\R} bV''d\xi=\frac{1}{2d} \int_{\R}W(\xi)d\xi>0$.  Substituting this expression in \eqref{eq06} gives
\begin{equation*}
\partial_\delta\K(\delta;k_2) =-\frac{\dis\int_{\R }Wd\xi}{2\delta\dis\int_\R aUVd\xi}<0.
\end{equation*}
So, \eqref{derivK} holds and $\delta \mapsto \K(\delta,k_2)$ is a decreasing bijection. \eqref{zeroarie} and \eqref{bound1} follow from Proposition \ref{propthreshold}. Now, it remains to establish the limits in \eqref{limit1}. 	These two endpoint limits are also contained in the recent work \cite[Theorem 1.3]{ma2026}. 

We first prove the limit at zero. We note that in \cite[Theorem 23 and Remark 24]{alzahrani2010}, the competition coefficients $\alpha$ and $\beta$ corresponds respectively to $k_1$ and $k_2$, their growth coefficient is equal to $1$ and their diffusion coefficient  $\varepsilon^2$ corresponds to $\delta$ here. Moreover, their front has the opposite orientation to ours. So, after reversing the sign of the speed, their result imply that  for every fixed $k_1\neq k_2^2$,
\begin{equation}\label{eq07}
	\sgn(c_{k_1,k_2,1,\delta})=\sgn(k_2^2-k_1),
\end{equation}
for all $\delta>0$ sufficiently small. Let $0<\eta<k_2^2-1$. Then applying \eqref{eq07} with $k_1=k_2^2-\eta>1$, we obtain  for all $\delta>0$ sufficiently small
\begin{equation*}
c_{k_2^2-\eta,k_2,1,\delta}>0.
\end{equation*}
Similarly, applying \eqref{eq07} with $k_1=k_2^2+\eta>1$ gives
\begin{equation*}
	c_{k_2^2+\eta,k_2,1,\delta}<0.
\end{equation*}
Thanks to \eqref{signK00}, we can deduce from the above identities that for all $\delta>0$ sufficiently small,
\begin{equation*}
k_2^2-\eta<\K(\delta;k_2)<k_2^2+\eta.
\end{equation*}
Since $\eta>0$ is arbitrary, we conclude that
\begin{equation}\label{ewen}
	\lim_{\delta\downarrow0}\K(\delta;k_2)=k_2^2.
\end{equation}

Now, we prove the limit at $+\infty$.  We first observe that since 	$\delta\mapsto\K(\delta;k_2)$ is strictly decreasing and bounded below by $\sqrt{k_2}$, the limit $\Lambda:=\dis \lim_{\delta\to+\infty}\mathcal K(\delta;k_2)$ exists and is such that  $\Lambda\geq\sqrt{k_2}>1$.  We let $\delta=\varepsilon^{-1}$ and we define define $\kappa_\varepsilon:=\K(\varepsilon^{-1};k_2)$.  Then $\dis\lim_{\epsilon\downarrow0}\kappa_\varepsilon=\Lambda$. Thanks to  \eqref{Kexchange}, we have $\K(\varepsilon;\kappa_\varepsilon)=k_2$. We claim that for all $\varepsilon>0$, the function $\beta \mapsto\K(\varepsilon;\beta)$ is strictly increasing.  Indeed, let $\beta_1<\beta_2$ and set $\alpha_1:=\K(\varepsilon;\beta_1)$. Then, $c_{\alpha_1, \beta_1,1,\varepsilon}=0$. Since the speed is strictly increasing with respect to $k_2$ \cite{kan1995}, we have $c_{\alpha_1, \beta_2,1,\varepsilon}>0$. We deduce from \eqref{signK00}, that $\alpha_1<\K(\varepsilon;\beta_2)$. Hence $\K(\varepsilon;\beta_1)<\K(\varepsilon;\beta_2)$. This proves the claim. Now, let $0<\eta<\Lambda-1$. For $\varepsilon>0$ sufficiently small, we have $\Lambda-\eta<\kappa_\varepsilon<\Lambda+\eta$. Thus, using the previous monotonicity, we obtain 
\begin{equation*}
\K(\varepsilon;\Lambda-\eta)<\K(\varepsilon;\kappa_\varepsilon)<\K(\varepsilon;\Lambda+\eta).
\end{equation*}
That is 
\begin{equation}\label{ewen1}
	\K(\varepsilon;\Lambda-\eta)<k_2<\K(\varepsilon;\Lambda+\eta).
\end{equation}
We note that $\Lambda-\eta, \Lambda+\eta>1$. Then, we take the limit as $\varepsilon\downarrow 0$ in \eqref{ewen1} and use \eqref{ewen} to deduce that
\begin{equation*}
(\Lambda-\eta)^2\leq k_2\leq(\Lambda+\eta)^2. 
\end{equation*}
Finally, letting $\eta\downarrow0$ in the previous inequality gives $\Lambda^2=k_2$.  Hence, $	\dis \lim_{\delta\to+\infty}\K(\delta;k_2)=\sqrt{k_2}.$ This completes the proof.
\end{proof}

We now prove the Theorem \ref{maintheorem2}. 
	
	\begin{proof}[Proof of Theorem \ref{maintheorem2}] 
		Firstly, if $k_1\leq\sqrt{k_2}$, then $k_1<\K(\delta;k_2)$ by Proposition \ref{propthreshold}, and so $c_{k_1,k_2,r,d}>0$ by \eqref{signK00}. Next, if $k_1\geq k_2^2$, then $k_1>\K(\delta;k_2)$ and hence $c_{k_1,k_2,r,d}<0$.  Now, assume that $\sqrt{k_2}<k_1<k_2^2$. By Theorem \ref{maintheorem1}, the function $\delta \mapsto \K(\delta;k_2)$ is a strictly decreasing bijection from $(0,\infty)$ onto $(\sqrt{k_2},k_2^2)$. Hence, there exist a unique  $\delta^*>0$ so that	$\mathcal K(\delta^*;k_2)=k_1$. Using \eqref{signK00}, we obtain	
		\begin{equation*}
			\c \left\{ 
			\begin{array}{llllll}
				>0, &  0<\frac{d}{r}<\delta^*,\\
				=0, & \frac{d}{r}=\delta^*,\\
				<0, &\frac{d}{r}>\delta^*.
			\end{array} 
			\right.
		\end{equation*}
		Finally, the front associated with $(k_1,k_2,\delta^*)$ is standing, so \eqref{signrel} gives
		$\sign(\delta^*-1)=\sign(k_2-k_1)$. This completes the proof.
	\end{proof}

We end this paper with the proof of Corollary \ref{corollarymonotonegeneral}.
\begin{proof}[Proof of Corollary \ref{corollarymonotonegeneral}]
Firstly, we assume that $k_1\geq k_2^2$.  Then Theorem \ref{maintheorem2} gives $\c<0$ for every $d>0$. If $d\geq 1$, then $\c\left(1-\frac{1}{d}\right) \leq 0$. We thus apply Theorem \ref{maintheorem} to deduce that $\partial_d\c<0$ for every $d \geq 1$. Thus $(i)$ holds. Now, assume that $\sqrt{k_2}<k_1<k_2^2$ and let $d^*:=r\delta^*(k_1,k_2)$.	Then by Theorem \ref{maintheorem2}, we have 
\begin{equation}\label{ewen3}
	\c
	\begin{cases}
		>0,&0<d<d^*,\\
		=0,&d=d^*,\\
		<0,&d>d^*.
	\end{cases}
\end{equation}
at $d=d*$, $\c=0$ and the assumptions of Theorem \ref{maintheorem} are satisfied, so $\partial_d\c|_{d=d^*}<0$.  Next, suppose that $d^*\geq 1$. Then for every $d>d^*$, $d\geq 1$ and by \eqref{ewen3} we obtain $\c<0$. In addition, $c_{k_1,k_2,r,d}\left(1-\frac{1}{d}\right) \leq 0$. So, Theorem \ref{maintheorem} gives $\partial_d\c<0$ for all $d\geq d^*$. Hence, $d\mapsto \c$ is strictly decreasing on $[d^*, \infty)$. Finally, suppose that $d^*<1$. Then for every $d \geq 1$, we have $d>d^*$ and \eqref{ewen3} implies $\c<0$. We have again $\c\left(1-\frac{1}{d}\right) \leq 0$. Theorem \ref{maintheorem} implies $\partial_d\c<0$ for all $d \geq 1$. Thus, $d\mapsto \c$ is strictly decreasing on $[1,\infty)$. This completes the proof. 
\end{proof}


\bibliographystyle{abbrv}
\bibliography{References}


\end{document}